\documentclass[10pt,reqno]{amsart}
\usepackage{amsmath}
\usepackage{amsthm}
\usepackage{amssymb}
\usepackage{graphicx}
\usepackage{amsfonts}
\usepackage{latexsym}
\usepackage{setspace}
\usepackage{enumitem}
\setlist[itemize]{leftmargin=*}
\setlist[enumerate]{leftmargin=*}

\usepackage{hyperref,caption,subcaption}
\usepackage{cite}

\usepackage{tikz}
\usetikzlibrary{positioning}
\usetikzlibrary{decorations.text}
\usetikzlibrary{decorations.pathmorphing}
\usetikzlibrary{shapes,arrows}

\tikzstyle{block:new} = [font = \scriptsize, draw, fill=green!20!white, text centered]
\tikzstyle{block:blue} = [font = \scriptsize, rectangle, draw, fill=blue!10!white, 
    text width=7em, text centered, minimum height=2em]        
\tikzstyle{block:red} = [font = \scriptsize, rectangle, draw, fill=red!20!white, 
    text width=7em, text centered, minimum height=2em]        
\tikzstyle{block:cyan} = [font = \scriptsize, rectangle, draw, fill=cyan!15!white, 
    text width=7em, text centered, minimum height=2em]

\usepackage{caption}
\tikzstyle{line} = [draw, -latex']

\usetikzlibrary{calc}

\theoremstyle{plain}
\numberwithin{equation}{section}
\newtheorem{theorem}[equation]{Theorem}

\newtheorem{lemma}[equation]{Lemma}

\theoremstyle{definition}

\newcommand{\F}{\mathbb{F}}
\newcommand{\N}{\mathbb{N}}

\newcommand{\Q}{\mathbb{Q}}
\newcommand{\Z}{\mathbb{Z}}
\renewcommand{\pmod}[1]{\,\,(\operatorname{mod} #1)}
\newcommand{\norm}[1]{\|#1\|}
\newcommand{\ndiv}{\!\nmid\!}

\renewcommand{\geq}{\geqslant}
\renewcommand{\leq}{\leqslant}

\let\oldenumerate=\enumerate
	\def\enumerate{
	\oldenumerate
	\setlength{\itemsep}{5pt}
	}
\let\olditemize=\itemize
	\def\itemize{
	\olditemize
	\setlength{\itemsep}{5pt}
	}

\allowdisplaybreaks

\begin{document}

\title[$p$-adic quotient sets II]{$p$-adic quotient sets II:\\quadratic forms}
	\author[C.~Donnay]{Christopher Donnay}
	
	\author[S.R.~Garcia]{Stephan Ramon Garcia}
	\address{Department of Mathematics, Pomona College, 610 N. College Ave., Claremont, CA 91711} 
	\email{stephan.garcia@pomona.edu}
	\urladdr{\url{http://pages.pomona.edu/~sg064747}}

	\author[J.~Rouse]{Jeremy Rouse}
        \address{Department of Mathematics and Statistics, Wake Forest University, 1834 Wake Forest Road, Winston-Salem, NC 27109}
        \email{rouseja@wfu.edu}
        \urladdr{\url{http://users.wfu.edu/rouseja}}

	\thanks{C.~Donnay partially supported by a fellowship from the University of Pennsylvania Graduate School of Education.  S.R.~Garcia partially supported by a David L.~Hirsch III and Susan H.~Hirsch Research Initiation Grant, the Institute for Pure and Applied Mathematics (IPAM) Quantitative Linear Algebra program, and NSF Grant DMS-1800123.}
	
\begin{abstract}
  For $A \subseteq \{1,2,\ldots\}$, we consider $R(A) = \{a/a' : a,a' \in A\}$.  If $A$ is the set of nonzero values assumed by a quadratic form, when is $R(A)$ dense in the $p$-adic numbers? We show that for a binary quadratic form $Q$,
  $R(A)$ is dense in $\Q_{p}$ if and only if the discriminant of $Q$ is a nonzero square in $\Q_{p}$, and for a quadratic form in at least three variables,
  $R(A)$ is always dense in $\Q_{p}$. This answers a question posed by several authors in 2017.
\end{abstract}

\keywords{$p$-adic number, quotient set, ratio set, quadratic form}

\maketitle

\section{Introduction}
For a subset $A \subseteq \N = \{1,2,3,\ldots\}$, let
$R(A) = \{a/a' : a,a' \in A\}$
denote the corresponding \emph{ratio set} (or \emph{quotient set}).  
The question of when $R(A)$ is dense in the positive
real numbers has been examined by many authors over the years
\cite{BT, QSDE, Hedman, Hobby, Nowicki, Misik, MR2505807, Pomerance, Micholson, Starni, 4QSG, Erdos, BST, Salat, SalatC, ST, STC}.  Analogues in the Gaussian integers \cite{QGP}
and, more generally, in algebraic number fields \cite{Sittinger}, have recently been considered.

The study of quotient sets in the $p$-adic setting was initiated by Florian Luca and the second author \cite{QFN}.
Shortly thereafter several other papers on the topic appeared \cite{pQS, MiskaMurruSanna, Sanna, MiskaSanna}.
In \cite{pQS} it was shown that if $A = \{ x^2 + y^2 : x,y\in  \Z \} \backslash\{0\}$,
then $R(A)$ is dense in $\Q_p$ if and only if $p \equiv 1 \pmod{4}$.
It is natural to wonder about possible extensions to other quadratic forms.

Fix a prime number $p$ and observe that each nonzero rational number has a unique representation of the form
$r = \pm p^k a/b$, in which $k \in \Z$, $a,b\in \N$, and $\gcd(a,p) = \gcd(b,p) = \gcd(a,b) =1$.  
The \emph{$p$-adic valuation} of such an $r$ is $\nu_p(r) = k$ and
its \emph{$p$-adic absolute value} is 
$\norm{r}_p = p^{-k}$.  By convention, $\nu_p(0) = \infty$ and $\norm{0}_p = 0$.  
The \emph{$p$-adic metric} on $\Q$ is $d(x,y) = \norm{x-y}_p$.
We write $\norm{\cdot}$ in place of $\norm{\cdot}_p$ when no confusion can arise.
The field $\Q_p$ of \emph{$p$-adic numbers} is the completion 
of $\Q$ with respect to the $p$-adic metric \cite{Gouvea,Koblitz}.  We let $\Q_p^{\times} = \Q_p \backslash \{0\}$.

A \emph{quadratic form} is a homogeneous polynomial 
\begin{equation}
  Q(x_{1},x_{2},\ldots,x_{r}) = \sum_{i=1}^{r} \sum_{j=i}^{r} a_{ij} x_{i} x_{j}.
\end{equation}
of degree $2$. We say that $Q$ is \emph{integral} if $a_{ij} \in \Z$
for all $i,j$, and we say that $Q$ is \emph{primitive} if there is no
positive integer $k > 1$ so that $k | a_{ij}$ for all $i$ and $j$.  We
can write $Q(\vec{x}) = \frac{1}{2} \vec{x}^{T} A \vec{x}$ for an $r \times r$
symmetric matrix $A$ (which will have even diagonal entries, and
integral off-diagonal entries). Two forms $Q$ and $Q'$
are \emph{equivalent} if there is an $r \times r$ matrix $M$
with integer entries and $\det(M) = \pm 1$ so that $Q'(\vec{x}) = Q(M \vec{x})$.

In the case of binary forms, we will distinguish \emph{proper} equivalence
(the case that $\det(M) = 1$) from \emph{improper} equivalence (the case that $\det(M) = -1$). Given a binary form
\begin{equation}
\label{eq:Q}
  Q(x,y) = ax^{2} + bxy + cy^{2},
\end{equation}
the \emph{discriminant} of $Q$ is $b^{2} - 4ac$. Equivalent binary
forms assume the same values and have the same discriminants.

Let $\F$ be a field. We say that $Q$ is \emph{nonsingular} over $\F$
if $\det(A) \ne 0$ (and \emph{singular} otherwise). We say that $Q$ is
\emph{isotropic} over $\F$ if there is a nonzero vector $\vec{x} \in \F^{r}$
so that $Q(\vec{x}) = 0$. Otherwise, $Q$ is \emph{anisotropic} over $\F$.
If $Q$ represents every value in $\F$, then $Q$ is \emph{universal} over $\F$.
It is known that if $Q$ is isotropic and nonsingular over $\F$, then
$Q$ is universal over $\F$ \cite[Thm.~I.3.4]{MR2104929}.

For brevity, the term ``quadratic form'' hereafter refers to a
quadratic form that is nonsingular over $\Q$, integral, and primitive. 
The quotient set generated by a quadratic form $Q$ is
\[
R(Q) = \{ Q(\vec{x})/Q(\vec{y}) : \vec{x}, \vec{y} \in \Z^{r}, Q(\vec{y}) \ne 0 \}.
\]
If $Q$ and $Q'$ are equivalent, then $R(Q) = R(Q')$. It has been asked when $R(Q)$ is dense in $\Q_p$ \cite[Problem 4.4]{pQS}. The main result of this paper
is a complete answer to this question.

\begin{theorem}
\label{Theorem:Main}
  Let $Q$ be an integral quadratic form in $r$ variables. Assume that
  $Q$ is primitive and is nonsingular over $\Q$ and let $p$ be a prime
  number.
\begin{enumerate}
\item If $Q$ is binary, then $R(Q)$ is dense in $\Q_{p}$ if and only if
  the discriminant of $Q$ is a square in $\Q_{p}$.
\item If $r \geq 3$, then $R(Q)$ is dense in $\Q_{p}$.
\end{enumerate}
\end{theorem}

We give two proofs of Theorem~\ref{Theorem:Main}a. Our first approach
is longer (Figure \ref{Figure:Tree}), but completely elementary.  The
second approach is shorter, but requires the classification of values
represented by quadratic forms over $\Q_{p}$ (as can be found in Serre's
book \cite{MR0344216}). This same tool is used to prove Theorem~\ref{Theorem:Main}b.

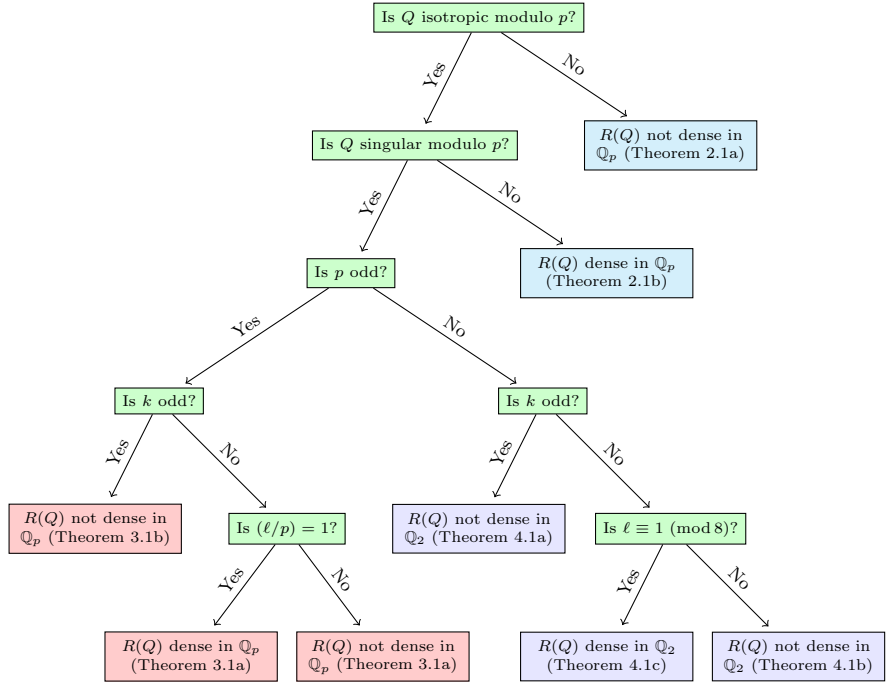
\begin{figure}
\centering
\begin{tikzpicture}[scale=0.85, every node/.style={scale=0.85}]
\node[block:new] (isotropic) at (0,0) {Is $Q$ isotropic modulo $p$?};

\node[block:new] (singular) at (-1,-2) {Is $Q$ singular modulo $p$?};
\node[block:cyan] (isotropic-end) at (3,-2) {$R(Q)$ not dense in $\Q_p$ (Theorem \ref{Theorem:Nonsingular}a)};

\node[block:new] (podd) at (-2,-4) {Is $p$ odd?};
\node[block:cyan] (singular-end) at (2,-4) {$R(Q)$ dense in $\Q_p$ (Theorem \ref{Theorem:Nonsingular}b)};

\node[block:new] (kodd) at (-5,-6) {Is $k$ odd?};
\node[block:new] (kodd2) at (1,-6) {Is $k$ odd?};

\node[block:red] (kodd-end) at (-6,-8) {$R(Q)$ not dense in $\Q_p$ (Theorem \ref{Theorem:Singular}b)};
\node[block:new] (pell) at (-3,-8) {Is $(\ell/p) = 1$?};

\node[block:red] (pell-end1) at (-4.5,-10) {$R(Q)$ dense in $\Q_p$ (Theorem \ref{Theorem:Singular}a)};
\node[block:red] (pell-end2) at (-1.5,-10) {$R(Q)$ not dense in $\Q_p$ (Theorem \ref{Theorem:Singular}a)};

\node[block:blue] (2odd-end) at (0,-8) {$R(Q)$ not dense in $\Q_2$ (Theorem \ref{Theorem:Two}a)};
\node[block:new] (2ell) at (3,-8) {Is $\ell \equiv 1 \pmod{8}$?};

\node[block:blue] (2ell-end1) at (2,-10) {$R(Q)$ dense in $\Q_2$ (Theorem \ref{Theorem:Two}c)};
\node[block:blue] (2ell-end2) at (5,-10) {$R(Q)$ not dense in $\Q_2$ (Theorem \ref{Theorem:Two}b)};

    \draw[->, shorten >= 0.1cm] (isotropic) -- (singular) node [midway, above, sloped] (TextNode) {\footnotesize Yes};
    \draw[->, shorten >= 0.1cm] (isotropic) -- (isotropic-end) node [midway, above, sloped] (TextNode) {\footnotesize No};
    \draw[->, shorten >= 0.1cm] (singular) -- (podd) node [midway, above, sloped] (TextNode) {\footnotesize Yes};
    \draw[->, shorten >= 0.1cm] (singular) -- (singular-end) node [midway, above, sloped] (TextNode) {\footnotesize No};
    \draw[->, shorten >= 0.1cm] (podd) -- (kodd) node [midway, above, sloped] (TextNode) {\footnotesize Yes};
    \draw[->, shorten >= 0.1cm] (podd) -- (kodd2) node [midway, above, sloped] (TextNode) {\footnotesize No};
    \draw[->, shorten >= 0.1cm] (kodd) -- (kodd-end) node [midway, above, sloped] (TextNode) {\footnotesize Yes};
    \draw[->, shorten >= 0.1cm] (kodd) -- (pell) node [midway, above, sloped] (TextNode) {\footnotesize No};
    \draw[->, shorten >= 0.1cm] (pell) -- (pell-end1) node [midway, above, sloped] (TextNode) {\footnotesize Yes};
    \draw[->, shorten >= 0.1cm] (pell) -- (pell-end2) node [midway, above, sloped] (TextNode) {\footnotesize No};
    \draw[->, shorten >= 0.1cm] (kodd2) -- (2odd-end) node [midway, above, sloped] (TextNode) {\footnotesize Yes};
    \draw[->, shorten >= 0.1cm] (kodd2) -- (2ell) node [midway, above, sloped] (TextNode) {\footnotesize No};
    \draw[->, shorten >= 0.1cm] (2ell) -- (2ell-end1) node [midway, above, sloped] (TextNode) {\footnotesize Yes};
    \draw[->, shorten >= 0.1cm] (2ell) -- (2ell-end2) node [midway, above, sloped] (TextNode) {\footnotesize No};
 \end{tikzpicture}
\caption{How to decide if $R(Q)$ is dense in $\Q_p$.
Here $Q$ is an integral, binary, and primitive quadratic form of discriminant $p^k \ell$, in which $\gcd(p,\ell)=1$.
Here $(\ell/p)$ denotes a Legendre symbol.}
\label{Figure:Tree}
\end{figure}

The organization of this paper is as follows. The elementary proof of
Theorem~\ref{Theorem:Main}a constitutes
sections~\ref{Section:Nonsingular}, \ref{Section:Singular}, and
\ref{Section:Two}. In Section \ref{Section:Nonsingular} we handle
binary quadratic forms that are nonsingular over $\F_{p}$; the results
therein apply to all primes. Section \ref{Section:Singular} concerns
binary quadratic forms that are singular modulo an odd prime and Section
\ref{Section:Two} treats forms that are singular modulo $2$. In Section \ref{Section:Fancy}, we give a more sophisticated proof of Theorem~\ref{Theorem:Main}a
as well as the proof of Theorem~\ref{Theorem:Main}b.

\section{Nonsingular (all primes)}\label{Section:Nonsingular}
Our aim in this section is to prove the following theorem, which addresses the two uppermost terminal nodes (blue) in Figure \ref{Figure:Tree}.

\begin{theorem}\label{Theorem:Nonsingular}
Let $Q(x,y)=ax^2+bxy+cy^2$ be primitive and integral. 
\begin{enumerate}
\item If $Q$ is anisotropic modulo $p$, then $R(Q)$ is not dense in $\Q_p$.
\item If $Q$ is isotropic and nonsingular modulo $p$, then $R(Q)$ is dense in $\Q_p$.
\end{enumerate}
\end{theorem}

\subsection{Proof of Theorem \ref{Theorem:Nonsingular}a}
Suppose that $Q$ is anisotropic over $\Z/p\Z$. 
We claim that $\nu_p(Q(x,y))$ is even for all $x,y\in\Z$.
If $Q(x,y) \not \equiv 0 \pmod{p}$, then $\nu_p(Q(x,y)) = 0$, which is even.
Suppose that $Q(x,y) \equiv 0 \pmod{p}$.  
Then $(x,y)\equiv (0,0)\pmod{p}$
since $Q$ is anisotropic; that is,
$x=mp^j$ and $y=np^k$, in which $j,k \geq 1$, $p \ndiv m$, and $p\ndiv n$.  
Without loss of generality, assume that $j\geq k$.  Then
\begin{align*}
\nu_p(Q(x,y)) 
& = \nu_p(am^2p^{2j}+bmnp^{j+k}+cn^2p^{2k} )\\
& = \nu_p\big(p^{2k}(am^2p^{2(j-k)}+bmnp^{j-k}+cn^2)\big)\\
& = 2k + \nu_p(Q(mp^{j-k}, n)) = 2k
\end{align*}
since $p\ndiv n$ and $Q$ is anisotropic.
Thus,
$
\nu_p(Q(x,y))-\nu_p(Q(z,w))\neq 1=\nu_p(p)
$
for all $x,y,z,w \in \Z$ and hence
$R(Q)$ is bounded away from $p$ in $\Q_p$.  Consequently,
$R(Q)$ is not dense in $\Q_p$.  \qed

\subsection{Proof of Theorem \ref{Theorem:Nonsingular}b for $p$ odd}
Before proceeding, we need two lemmas.

\begin{lemma}[Lemma 2.3 of \cite{pQS}]\label{Lemma:N}
Let $A\subset \N$ and let $p$ be a prime.
\begin{enumerate}
\item If $A$ is $p$-adically dense in $\N$, then $R(A)$ is dense in $\Q_p$.
\item $R(A)$ is $p$-adically dense in $\N$ if and only if $R(A)$ is dense in $\Q_p$.
\end{enumerate}
\end{lemma}

\begin{proof}
\noindent(a) 
If $A$ is $p$-adically dense in $\N$, it is $p$-adically dense in $\Z$.
Inversion is continuous on $\Q_p^{\times}$, so $R(A)$ is $p$-adically dense in $\Q$,
which is dense in $\Q_p$.

\noindent(b) 
Suppose that $R(A)$ is $p$-adically dense in $\N$.
Since inversion is continuous on $\Q_p^{\times}$, the result
follows from the fact that $\N$ is $p$-adically dense in $\{x \in \Q : \nu_p(x) \geq 0\}$.
\end{proof}

\begin{lemma}\label{Lemma:PartialInvertible}
Let $Q$ be nonsingular modulo an odd prime $p$.  
If $(x,y)\not\equiv (0,0)\pmod{p}$ and $Q(x,y)\equiv 0\pmod{p}$,
then $2ax+by \not\equiv 0\pmod{p}$ or $bx+2cy \not\equiv 0\pmod{p}$.
\end{lemma}

\begin{proof}
We prove the contrapositive.  Suppose that
\begin{equation}\label{eq:2axby}
2ax+by \equiv bx+2cy\equiv 0\pmod{p}.
\end{equation} 
Since $Q$ is nonsingular, $b^2 \not \equiv 4ac \pmod{p}$.
If $p|b$, then $p\ndiv a$ and $p\ndiv c$.  Thus, there are two cases:
$p\ndiv a$ and $p\ndiv c$, or $p\ndiv b$.
\medskip

\noindent\textsc{Case 1}: If $p\ndiv a$ and $p\ndiv c$, then \eqref{eq:2axby} implies that
\begin{equation}\label{eq:xby2a}
x \equiv - \frac{by}{2a} \pmod{p}.
\end{equation}
Thus, 
\begin{equation*}
0 \equiv Q\left( - \frac{by}{2a}, y \right)\equiv \left(\frac{-b^2+4ac}{4a}\right) y^2\pmod{p}
\end{equation*}
and hence $y \equiv 0 \pmod{p}$.  Then \eqref{eq:xby2a} implies that
$(x,y) \equiv (0,0) \pmod{p}$.
\medskip

\noindent\textsc{Case 2}: If $p\ndiv b$, then 
\begin{equation}\label{eq:x2cyb}
x \equiv -\frac{2cy}{b} \pmod{p}
\end{equation}
and hence
\begin{equation*}
0 \equiv Q\left( -\frac{2cy}{b},y\right) \equiv- cy^2\left(\frac{b^2-4ac}{b^2}\right) \pmod{p}.
\end{equation*}
Consequently, $p| y$ or $p| c$. 
\begin{itemize}
\item If $p|y$, then \eqref{eq:x2cyb} implies that $(x,y) \equiv (0,0) \pmod{p}$.
\item If $p| c$, then \eqref{eq:x2cyb} implies that $p|x$.
Since $p \ndiv b$, \eqref{eq:2axby} ensures that $p|y$.
Thus, $(x,y) \equiv (0,0) \pmod{p}$. \qedhere
\end{itemize}
\end{proof}

Suppose that $Q$ is isotropic and nonsingular modulo an odd prime $p$.
By Lemma \ref{Lemma:N}, it suffices to show that for each $n \in \Z$ and $r\geq 1$, there exists an $(x,y) \in \Z^2$ 
such that 
$Q(x,y)\equiv n \pmod{p^r}$.
To this, we add the requirement 
\begin{equation}\label{eq:ProofConditions}
p\ndiv \left(2ax+by\right) \qquad \text{or} \qquad p\ndiv \left(bx+2cy\right).
\end{equation}
We induct on $r$.
The base case is $r=1$.  
\begin{itemize}
\item If $n \equiv 0 \pmod{p}$, then since $Q$ is isotropic
we may find $(x,y) \not \equiv (0,0) \pmod{p}$ so that $Q(x,y) \equiv 0 \pmod{p}$.
Lemma \ref{Lemma:PartialInvertible} ensures that at least one of the two conditions in \eqref{eq:ProofConditions} hold.

\item If $n \not \equiv 0 \pmod{p}$, then there is an $(x,y)$ so that $Q(x,y) \equiv n \pmod{p}$
since $Q$ is isotropic and nonsingular \cite[Prop.~3.4]{MR2104929}.  Since $p$ is odd, 
\begin{equation*}
0 \not \equiv n \equiv Q(x,y) 
\equiv \frac{x}{2}(2ax+by) + \frac{y}{2}(bx+2cy) \pmod{p},
\end{equation*}
which implies that \eqref{eq:ProofConditions} holds.
\end{itemize}

Now suppose that $Q(x,y) \equiv n \pmod{p^r}$ and, without loss of generality, that $p \ndiv (2ax+by)$.  Then
$Q(x,y)=n+mp^r$ for some $m\in\Z$.  If 
\begin{equation*}
i\equiv -(2ax+by)^{-1}m\pmod{p},
\end{equation*} 
then the identity
\begin{equation}\label{eq:QuadraticIdentity}
Q(x+z,y)=Q(x,y)+az^2+bzy+2axz
\end{equation}
yields
\begin{align*}
Q(x+ip^r,y) & = Q(x,y)+ai^2p^{2r}+bip^ry+2axip^r\\
& = n+mp^r+ai^2p^{2r}+bip^ry+2axip^r \\
&\equiv n + mp^r+bip^ry +2axip^r \pmod{p^{r+1}}\\
& \equiv n+p^r(m+(2ax+by)i) \pmod{p^{r+1}}\\
&\equiv n \pmod{p^{r+1}},
\end{align*}
in which $2a(x+ip^r) +by = (2ax+by) + 2aip^r$ is not divisible by $p$.
This completes the induction. \qed

\subsection{Proof of Theorem \ref{Theorem:Nonsingular}b for $p =2$}
Suppose that $Q$ is isotropic and nonsingular modulo $2$.
Since $2 \ndiv (b^2-4ac)$, it follows that $b$ is odd and hence
\begin{equation*}
Q(x,y) \equiv ax^2 + xy + cy^2 \pmod{2}.
\end{equation*}
Because $Q$ is isotropic, $a$ or $c$ is even; see Table \ref{Table:Qm2}.
Without loss of generality, suppose that $a$ is even.  
By Lemma \ref{Lemma:N}, it suffices to show that for each $n \in \Z$ and $r \geq 1$, there is an 
$(x,y) \in \Z^2$ such that 
\begin{equation}\label{eq:Qxynpr2}
Q(x,y) \equiv n \pmod{2^r} 
\qquad\text{and}\qquad
y \not \equiv 0 \pmod{2}.
\end{equation}
We proceed by induction on $r$.
For the base case $r=1$, we may let $(x,y) = (n-c,1)$.

\begin{table}
\begin{equation*}
\begin{array}{cc|ccccc}
x & y & Q(x,y) \pmod{2} & \text{$a,c$ odd} & \text{$a,c$ even} & \text{$a$ even, $c$ odd}  & \text{$a$ odd, $c$ even} \\
\hline
0 & 0 & 0 & 0 & 0 & 0 & 0 \\
0 & 1 & c & 1 & 0 & 1 &0\\
1 & 0 & a & 1 & 0 & 0&1\\
1 & 1 & 1+a+c & 1& 1 & 0&0\\
\end{array}
\end{equation*}
\caption{Values of $Q(x,y)\equiv ax^2+xy +cy^2 \pmod{2}$.}
\label{Table:Qm2}
\end{table}

Now suppose that \eqref{eq:Qxynpr2} holds for some $r$.  Then
$Q(x,y)=n+m2^r$ for some $m\in \Z$.  If
$i\equiv mb^{-1}y^{-1}\pmod{2}$,
then \eqref{eq:QuadraticIdentity} yields
\begin{align*}
Q(x+2^r i ,y) 
&= Q(x,y) + a(2^r i )^2 + b(2^r i )y + 2ax(2^r i ) \\
& = (n +m2^r) +2^{2r}ai^2+2^r biy+2^{r+1}aix\\
&\equiv n+m2^r +2^r bi y \pmod{2^{r+1}}\\
& \equiv n+2^r(m+biy) \pmod{2^{r+1}}\\
&\equiv n \pmod{2^{r+1}}.
\end{align*}
This completes the induction. \qed

\section{Singular modulo an odd prime}\label{Section:Singular}
Our aim in this section is to prove the following theorem, which addresses the three lower-left 
terminal nodes (red) in Figure \ref{Figure:Tree}.  Below $(\ell/p)$ is a Legendre symbol.

\begin{theorem}\label{Theorem:Singular}
Let $Q(x,y)=ax^2+bxy+cy^2$ be primitive and integral with
discriminant $p^k \ell$, in which $k \geq 1$ and $p$ is an odd prime that does not divide $\ell$.
\begin{enumerate}
\item If $k$ is even, then $R(Q)$ is dense in $\Q_p$ if and only if $(\ell/p) = 1$.
\item If $k$ is odd, then $R(Q)$ is not dense in $\Q_p$.
\end{enumerate}
\end{theorem}

\subsection{Proof of Theorem \ref{Theorem:Singular}a}
We have $b^2-4ac = p^k \ell$ with $k \geq 2$ even.
Because $Q$ is primitive, $p$ cannot divide both $a$ and $c$ since otherwise it would divide 
$a$, $b$, and $c$.  
Without loss of generality, suppose that $p \ndiv a$.  
Let $u \equiv 2^{-1}a^{-1}b \pmod{p^k}$, so that $2ua-b\equiv 0 \pmod{p^k}$. 
The forms $Q(x,y)$ and
\begin{equation*}
Q'(x,y) = Q(-x-uy,y) = 
ax^2+(2ua-b)xy+(u^2a-ub+c)y^2
\end{equation*}
are (improperly) equivalent.  Thus, $Q$ and $Q'$
have the same discriminant
and assume the same values, hence $R(Q) = R(Q')$.  
Since $p \ndiv 4a$ and 
\begin{equation*}
4a(u^2a-ub+c)
\equiv (2au-b)^2-(b^2-4ac)
\equiv 0 \pmod{p^k},
\end{equation*}
it follows that
\begin{equation*}
B = \frac{2ua-b}{p^{k/2}}
\qquad\text{and}\qquad
C=\frac{u^2a-ub+c}{p^k}
\end{equation*}
are integers.  We may write
\begin{equation}\label{eq:QSP}
Q'(x,y)=ax^2+p^{k/2}Bxy+p^kCy^2,
\end{equation}
which has discriminant
\begin{equation*}
p^k (B^2 - 4 aC) = p^k \ell.
\end{equation*}
Consequently, the integral quadratic form
\begin{equation}\label{eq:QDP}
Q''(x,y) = ax^2+Bxy+Cy^2
\end{equation}
has discriminant $B^2 - 4aC = \ell$. Moreover,
\begin{equation}\label{eq:DamnFourThing}
4aQ''(x,y) = (2ax+By)^2- \ell y^2.
\end{equation}

\medskip\noindent\textsc{Case 1}:
Suppose that $(\ell/p)=-1$. 
If $Q''(x_0,y_0) \equiv 0 \pmod{p}$, then \eqref{eq:DamnFourThing} implies
\begin{equation*}
\left(2ax_0+By_0\right)^2  \equiv \ell y_0^2 \pmod{p}
\end{equation*}
since $p\ndiv 4a$.
The Legendre symbol of the left-hand side is $0$ or $1$; the Legendre symbol of the right-hand side
is $0$ or $-1$.
Thus, both sides are congruent to $0$ modulo $p$ and hence $y_0 \equiv 0 \pmod{p}$.  Since 
$p\ndiv 2a$, it follows that $x_0 \equiv 0 \pmod{p}$
and hence $Q''$ is anisotropic modulo $p$.
Theorem \ref{Theorem:Nonsingular} ensures that $R(Q'')$ is not dense in $\Q_p$.
Since
$Q'(x,y) = Q''(x,p^{k/2}y)$,
we conclude that $R(Q')$, which equals $R(Q)$, is not dense in $\Q_p$.

\medskip\noindent\textsc{Case 2}:
Suppose that $(\ell/p)=1$. 
Let $\sqrt{\ell}$ denote a square root of $\ell$ modulo $p$ and
let $(x_0,y_0) \equiv (\sqrt{\ell}-B,2a)\pmod{p}$, which is not congruent to $(0,0)$ modulo $p$ since $p\ndiv 2a$.
Then \eqref{eq:QDP} yields
\begin{equation*}
4a Q''(x_0,y_0)
\equiv \big(2a (\sqrt{\ell}-B) + B(2a)\big)^2 - \ell (2a)^2 
\equiv a \ell - a \ell 
\equiv 0 \pmod{p}.
\end{equation*}
Since $p\ndiv 4a$, it follows that 
$Q''$ is isotropic modulo $p$.  Since the discriminant $\ell$ of $Q''$ is not divisible by $p$,
Theorem \ref{Theorem:Nonsingular}b implies that $R(Q'')$ is dense in $\Q_p$.
If $Q''(z,w) \neq 0$, then \eqref{eq:QSP} provides
\begin{equation*}
\frac{Q''(x,y)}{Q''(z,w)}
= \frac{p^k Q''(x,y)}{p^k Q''(z,w)}
= \frac{Q''(p^{k/2}x,p^{k/2}y)}{Q''(p^{k/2}z,p^{k/2}w)} 
= \frac{Q'(p^{k/2}x,y)}{Q'(p^{k/2}z,w)},
\end{equation*}
and hence $R(Q')$ is dense in $\Q_p$.  Since $Q$ and $Q'$ are equivalent,
$R(Q) = R(Q')$ is also dense in $\Q_p$.  \qed

\subsection{Proof of Theorem \ref{Theorem:Singular}b}

As in the proof of Theorem \ref{Theorem:Singular}a, 
we may assume that $p\ndiv a$.  Since $R(Q) = R(4aQ)$ and
\begin{equation*}
4aQ(x,y) = \left(2a x + b y\right)^2 - (b^2-4ac) y^2,
\end{equation*}
we may assume without loss of generality that
\begin{equation*}
Q(x,y) = x^2 -p^k \ell y^2.
\end{equation*}
Suppose toward a contradiction that 
$R(Q)$ is dense in $\Q_p$. Let $n$ be a quadratic nonresidue modulo $p$.  
Then there are $x,y,z,w \in \Z$, not all multiples of $p$, so that $Q(z,w) \neq 0$ and
\begin{equation}\label{eq:QQhyppk}
\left\| \frac{Q(x,y) }{Q(z,w) } - n \right\| < \frac{1}{p^{k}}.
\end{equation}
In particular, $\|Q(x,y)\| = \|Q(z,w)\|$.  Multiplying \eqref{eq:QQhyppk}
by $Q(z,w)$ gives 
\begin{equation*}
\left\| (x^{2} - nz^{2}) - p^{k} \ell (y^{2} - nw^{2})\right\| = \left\| Q(x,y) - nQ(z,w) \right\| < \frac{\|Q(z,w)\|}{p^{k}} \leq \frac{1}{p^{k}}.
\end{equation*}
If $p \ndiv x$ or $p \ndiv z$, then $x^{2} - nz^{2} \not\equiv 0 \pmod{p}$
and hence $\left\|Q(x,y) - nQ(z,w)\right\| = 1$, which is a contradiction.

Since $p | x$ and $p | z$, we get $p \ndiv y$ or $p
\ndiv w$.  Thus, $y^{2} - nw^{2} \not\equiv 0
\pmod{p}$. 
Now observe that $x^{2}-nz^{2}$ has even $p$-adic valuation (the form
$u^2-nv^2$ is anisotropic and nonsingular modulo $p$ and the proof of
Theorem \ref{Theorem:Nonsingular}a ensures that it has even $p$-adic valuation for all $u,v$).
Consequently, $Q(x,y) - nQ(z,w)$ is the sum of a $p$-adic
integer with even valuation, and one with odd valuation $k$.
Thus, $\|Q(x,y) - nQ(z,w)\| \geq p^{-k}$, which is a
contradiction. Since $n$
cannot be arbitrarily well approximated by elements of
$R(Q)$, it follows that $R(Q)$ is not dense in $\Q_p$. \qed

\section{Singular modulo $2$}\label{Section:Two}
Our aim in this section is to prove the following theorem, which addresses the three lower-right 
terminal nodes (purple) in Figure \ref{Figure:Tree}. 

\begin{theorem}\label{Theorem:Two}
Let $Q(x,y)=ax^2+bxy+cy^2$ be primitive and integral with discriminant $2^k \ell$, in which $\ell$ is odd.
\begin{enumerate}
\item If $k$ is odd, then $R(Q)$ is not dense in $\Q_2$.
\item If $k$ is even and $\ell\not\equiv 1\pmod{8}$, then $R(Q)$ is not dense in $\Q_2$.
\item If $k$ is even and $\ell \equiv 1 \pmod{8}$, then $R(Q)$ is dense in
$\Q_{2}$.
\end{enumerate}
\end{theorem}

\subsection{Proof of Theorem \ref{Theorem:Two}a}
The proof is similar in flavor to that of Theorem \ref{Theorem:Singular}b, although there are a couple modifications. Since $R(Q) = R(4aQ)$ and
$4aQ(x,y) = (2ax+by)^{2} - (b^{2}-4ac)y^{2}$, we may assume without loss
of generality that $Q(x,y) = x^{2} - 2^{k} \ell y^{2}$. Suppose that $R(Q)$
is dense in $\Q_{2}$. Then there are $x, y, z, w \in \Z$, not all even, so that
$Q(z,w) \ne 0$ and
\begin{equation*}
\left\| \frac{Q(x,y)}{Q(z,w)} - 5 \right\| < \frac{1}{2^{k+2}}.
\end{equation*}
We also see that $\|Q(x,y)\| = \|Q(z,w)\|$ and from
this we get
\begin{equation}\label{eq:ContQ2}
 \|(x^{2} -5z^{2}) - 2^{k} \ell (y^{2} - 5w^{2})\| = \|Q(x,y) - 5Q(z,w)\|  < \frac{1}{2^{k+2}}.
\end{equation}
If $x$ or $z$ is odd, then $x^{2} - 5z^{2} \equiv 1,3, \text{or}\, 4 \pmod{8}$. It follows that
the power of $2$ dividing $x^{2} - 5z^{2}$ is even.
If $x$ and $z$ are odd, then $\|Q(x,y) - 5Q(z,w)\|
\geq 1/4$, which contradicts \eqref{eq:ContQ2}.
Thus, $x$ and $z$ are both even. However, in this case,
the power of $2$ dividing $x^{2} - 5z^{2}$ is even,
and the power of $2$ dividing $2^{k} \ell (y^{2} - 5w^{2})$ is odd
and at most $2^{k+2}$. It follows that
\begin{equation*}
\left\| Q(x,y) - 5Q(z,w) \right\| \geq \frac{1}{2^{k+2}},
\end{equation*}
which is a contradiction. Thus, $R(Q)$ is not dense is $\Q_{2}$. \qed

\subsection{Proof of Theorem \ref{Theorem:Two}b}

In this section, we show that if $b^{2} - 4ac = 2^{k} \ell$ with $k$
even and $\ell \equiv 3, 5 \text{ or } 7 \pmod{8}$, then $R(Q)$ is not
dense in $\Q_{2}$.  As before, if $Q = ax^{2} + bxy + cy^{2}$, then
$R(Q) = R(4aQ) = (2ax+by)^{2} - (b^{2} - 4ac)y^{2}$ and so if
$Q'(x,y) = x^{2} - 2^{k} \ell y^{2}$, then $R(Q)
\subseteq R(Q')$.  Letting $Q''(x,y) = x^{2} - \ell y^{2}$,
we have  
\begin{equation*}
  \frac{Q'(x,y)}{Q'(z,w)} = \frac{Q''(x,2^{k/2} y)}{Q''(z,2^{k/2}w)}
\end{equation*}
for $x, y, z, w \in \Z$
and hence $R(Q') \subseteq R(Q'')$. Consequently, it suffices to show that $R(Q'')$
is not dense in $\Q_{2}$.  We require a couple computational lemmas.

\begin{lemma}
If $\ell \equiv 5 \pmod{8}$, then $R(Q'')$ is not dense in $\Q_2$.
\end{lemma}

\begin{proof}
Write $x = 2^j \tilde{x}$ and $y = 2^k \tilde{y}$, in which $j, k\geq 0$ and $\tilde{x}, \tilde{y}$ are odd.
\begin{itemize}
\item If $j < k$, then
\begin{align*}
\nu_2(Q''(x,y))
&= \nu_2 \big( (2^j \tilde{x})^2 -  \ell (2^k \tilde{y})^2 \big) \\
&= \nu_2\big( 2^{2j} \tilde{x}^2 - 2^{2k} \ell \tilde{y}^2 \big)\\
&=2j + \nu_2(\tilde{x}^2 - 2^{2(k-j)} \ell \tilde{y}^2) \\
&= 2j.
\end{align*}

\item If $j > k$, then
\begin{align*}
\nu_2(Q''(x,y))
&= \nu_2 \big( (2^j \tilde{x})^2 -  \ell (2^k \tilde{y})^2 \big) \\
&= \nu_2\big( 2^{2j} \tilde{x}^2 - 2^{2k} \ell \tilde{y}^2 \big)\\
&=2k + \nu_2(2^{2(j-k)}\tilde{x}^2 -\ell \tilde{y}^2) \\
&= 2k.
\end{align*}

\item If $j = k$, then
\begin{align*}
\nu_2(Q''(x,y))
&= \nu_2 \big( (2^j \tilde{x})^2 -  \ell (2^k \tilde{y})^2 \big) \\
&= 2j + \nu_2\big( \tilde{x}^2 - \ell \tilde{y}^2 \big).
\end{align*}
If $\ell \equiv 5 \pmod{8}$, then
\begin{equation*}
\tilde{x}^2 - \ell \tilde{y}^2 \equiv 4 \pmod{8}
\end{equation*}
since $\tilde{x}^2\equiv \tilde{y}^2 \equiv 1 \pmod{8}$.
Thus, $\nu_2(Q''(x,y))$ is even.
\end{itemize}
It follows that $\nu_{2}(Q''(x,y)/Q''(z,w))$ is even, and so
there are no solutions to
\begin{equation*}
\left\| \frac{Q''(x,y)}{Q''(z,w)} - 2 \right\| < \frac{1}{2}.
\end{equation*}
Thus, $R(Q'')$ is not dense in $\Q_{2}$.
\end{proof}

\begin{lemma}
  If $\ell \equiv 3 \text{ or } 7 \pmod{8}$, then $R(Q'')$ is not dense in $\Q_2$.
\end{lemma}

\begin{proof}
  Suppose that $R(Q'')$ is dense in $\Q_{2}$. Then
  there are $x,y,z,w \in \Z$ so that
  \begin{equation*}
\left\| \frac{Q''(x,y)}{Q''(z,w)} - 3 \right\| \leq \frac{1}{2^{3}}.
\end{equation*}
We may assume at least
one of $x,y,z,w$ is odd. Multiplying by $\|Q''(z,w)\|$ gives
\begin{equation*}
\left\|(x^{2} - \ell y^{2}) - 3(z^{2} - \ell w^{2})\right\| \leq \frac{1}{2^{3}}.
\end{equation*}
For $\ell = 3$, a computation confirms that there are no solutions
to $x^{2} - 3y^{2} - 3z^{2} + 9w^{2} \equiv 0 \pmod{8}$ with at least one of
$x,y,z,w$ is odd. 
For $\ell = 7$, there are no solutions
to $x^{2} - 7y^{2} - 3z^{2} + 21w^{2} \equiv 0 \pmod{8}$ with at least one of $x,y,z,w$ is odd. 
This contradiction tells us that
$R(Q'')$ is not dense in $\Q_{2}$.
\end{proof}

\subsection{Proof of Theorem \ref{Theorem:Two}c}

Suppose that $Q(x,y) = ax^{2} + bxy + cy^{2}$ is primitive and $b^{2}
- 4ac = 2^{k} \ell$ where $k \geq 2$ is even and $\ell \equiv 1
\pmod{8}$.  Since $b^{2} - 4ac \equiv 0 \pmod{4}$,
$b$ must be even. By switching $a$ and $c$ if necessary, we may assume that $a$ is odd.
The form $Q(x,y)$ is equivalent
to $$Q'(x,y) = Q(x+qy,y) = ax^{2} + (2aq+b)xy + (aq^{2} + bq + c)y^{2}$$
and hence $R(Q) = R(Q')$. We claim that we can choose a $q$
such that 
\begin{equation}\label{eq:Des2}
2aq+b \equiv 0 \pmod{2^{k/2}} \quad \text{and}\quad
aq^{2} + bq + c \equiv 0 \pmod{2^{k}}.
\end{equation}
Let $$q \equiv -\frac{b}{2a}  + 2^{k/2 - 1} \pmod{2^{k}}.$$
Then
\begin{equation*}
2aq +b \equiv (2a) \bigg(-\frac{b}{2a}  + 2^{k/2 - 1}\bigg) + b
\equiv -b + a 2^{k/2 } + b 
\equiv 0 \pmod{2^{k/2}},
\end{equation*}
which is the first condition in \eqref{eq:Des2}.  The second condition follows from
\begin{align*}
aq^{2} + bq + c 
&\equiv a \bigg(-\frac{b}{2a}  + 2^{k/2 - 1}\bigg)^{2} + b\bigg(-\frac{b}{2a}  + 2^{k/2 - 1}\bigg) + c \\
&\equiv 2^{k-2}a - \frac{b^2}{4a} + c \\
&\equiv 2^{k-2}a - \frac{2^k \ell + 4ac}{4a} + c \\
&\equiv \frac{1}{a}\big( 2^{k-2} a^2 - 2^{k-2} \ell \big) \pmod{2^k} \\
& \equiv 0 \pmod{2^k}
\end{align*}
since $a$ is odd and $\ell \equiv 1 \pmod{8}$.  Thus, we may define the integers
\begin{equation*}
B = \frac{2aq+b}{2^{k/2}} \qquad \text{and} \qquad C = \frac{c+bq+aq^{2}}{2^{k}}
\end{equation*}
so that the form $$Q''(x,y) = ax^{2} + Bxy + Cy^{2}$$ has discriminant
\begin{equation}\label{eq:BaCl}
B^{2} - 4aC = \frac{(2aq+b)^{2} - 4a (c+bq+aq^{2})}{2^{k}} = \frac{b^{2} - 4ac}{2^{k}} = \ell \equiv 1 \pmod{8}.
\end{equation}
Since $Q'(x,y) = Q''(x,2^{k/2} y)$, we have
$R(Q') \subseteq R(Q'')$. Since
\begin{equation*}
\frac{Q'(x,y)}{Q'(z,w)} = \frac{Q'(2^{k/2} x, 2^{k/2} y)}{Q'(2^{k/2} z, 2^{k/2} w)} = \frac{Q''(2^{k/2} x, y)}{Q''(2^{k/2} z, w)},
\end{equation*}
we get $R(Q'') \subseteq R(Q')$.  Thus, $R(Q') = R(Q'')$. 

From \eqref{eq:BaCl}, it follows that $B$ is odd
and hence $B^{2} \equiv 1 \pmod{8}$. Thus, either $a$ or $C$ is even
and it follows that $Q''$ is isotropic modulo $2$. 
Theorem~\ref{Theorem:Nonsingular}b ensures that $R(Q'')$ is dense in $\Q_{2}$.
\qed

\section{An alternative approach}\label{Section:Fancy}

In this section, we present an alternative approach to the proof of
Theorems~\ref{Theorem:Nonsingular}, \ref{Theorem:Singular}, and
\ref{Theorem:Two}. We also prove that if $Q$ is a non-degenerate
quadratic form in $r \geq 3$ variables, then $R(Q)$ is dense in
$\Q_{p}$ for all $p$. While the arguments given here are shorter, they rely
heavily on the classification of quadratic forms over $\Q_{p}$ and the
values they represent. One convenient source for this material is
\cite{MR0344216}.

Over a field, any quadratic form $Q$ is equivalent to a diagonal one (by
\cite[Thm.~IV.1]{MR0344216}), namely
\[
Q' = a_{1} x_{1}^{2} + a_{2} x_{2}^{2} + \cdots + a_{r} x_{r}^{2}.
\]
For the remainder of this section, we will use the classification of
squares in $\Q_{p}$ (see 
\cite[Thms. 2.3 \& 2.4]{MR0344216}).  If $p > 2$, then an element $x = p^{n} u \in
\Q_{p}$ with $u \in \Z_{p}$ and $\nu_{p}(u) = 0$ is a square if and
only if $n$ is even and $u \bmod p \in \F_{p}$ is a square.
If $p = 2$, then an element $x = 2^{n} u \in \Q_{2}$ is a square if and only
if $n$ is even and $u \equiv 1 \pmod{8}$. It follows from
this that $\Q_{p}^{\times}$ has four square classes if $p > 2$
and eight square classes if $p = 2$.

The corollary on page 37 of \cite{MR0344216} gives a classification of
the values reprsented by a quadratic form over $\Q_{p}$. We wish to
record some consequences of this corollary. In particular, a binary
quadratic form over $\Q_{p}$ whose discriminant is not a square
represents half of the square classes, while a binary quadratic form
over $\Q_{p}$ whose discriminant is a square represents everything in
$\Q_{p}$. A quadratic form in three variables either represents
everything in $\Q_{p}$, or represents all but one square class. Finally,
a quadratic form in four or more variables over $\Q_{p}$ is universal.

We begin by reproving Theorems~\ref{Theorem:Nonsingular},
\ref{Theorem:Singular}, and \ref{Theorem:Two}.  We start with a result
of Arnold (which he attributes to F.~Aicardi) \cite[Thm.~1]{Arnold}.
\begin{lemma}[Arnold]
\label{mult}
Let $Q(x,y) = ax^{2} + bxy + cy^{2}$ be a binary quadratic form with integer
coefficients.
If $Q$ represents $A$, $B$ and $C$, then it represents $ABC$.
\end{lemma}

One way of interpreting this statement is that the inverse of $Q(x,y)
= ax^2+bxy+cy^2$ in the class group is $ax^{2} - bxy + cy^{2}$, which
is improperly equivalent to $Q$.  Because $Q \circ Q^{-1} \circ Q = Q$
in the class group, if $Q$ represents $A$, $Q^{-1}$ represents $B$,
and $Q$ represents $C$, then $Q = Q \circ Q^{-1} \circ Q$ represents
$ABC$.

\begin{proof}
If $Q(x_{1},y_{1}) = A$, $Q(x_{2},y_{2}) = B$ and $Q(x_{3},y_{3}) = C$, then
$Q(x,y) = Q(x_{1},y_{1}) Q(x_{2},y_{2}) Q(x_{3},y_{3})$, in which
\begin{align*}
  x &= (ax_{1}x_{2} - c y_{1}y_{2})x_{3} + (c(y_{1}x_{2} + x_{1}y_{2}) + bx_{1}x_{2})y_{3}\\
  y &= (a(x_{1}y_{2} + x_{2}y_{1}) + b y_{1}y_{2})x_{3} + (-ax_{1} x_{2} + cy_{1}y_{2})y_{3}.\qedhere
\end{align*}
\end{proof}

The following result provides an alternate representation of $R(Q)$ based upon Arnold's lemma.

\begin{lemma}\label{quotset}
Let $Q$ be a binary quadratic form and let $a$ be a nonzero integer represented by $Q$. Then
\begin{equation*}
  R(Q) = \left\{ \frac{Q(x,y)}{a} : x, y \in \Q \right\}.
\end{equation*}
\end{lemma}

\begin{proof}
Suppose that $b = Q(x,y)/a$,  in which $x,y \in Q$ and $a =Q(z,w)$ for some $z,w \in \Z$.
Write $x = c/f$ and $y = d/f$, in which
$c,d,f \in \Z$ and $f \neq 0$. Then
\begin{equation*}
b = \frac{Q(c/f,d/f)}{a} = \frac{ Q(c,d) / f^2}{a}
= \frac{Q(c,d)}{af^{2}} = \frac{Q(c,d)}{Q(fz,fw)} \in R(Q).
\end{equation*}
Now suppose that $b \in R(Q)$. Then there are $x_1,y_1,x_2,y_2 \in \Q$ so that 
$$b = \frac{Q(x_{1},y_{1})}{Q(x_{2},y_{2})} = \frac{a Q(x_{1},y_{1}) Q(x_{2},y_{2})}{a Q(x_{2},y_{2})^{2}}.$$ By Lemma~\ref{mult}, there are $X, Y \in \Z$
so that $Q(X,Y) = a Q(x_{1},y_{1}) Q(x_{2},y_{2})$. Thus,
\begin{equation*}
b = \frac{Q(X/Q(x_{2},y_{2}),Y/Q(x_{2},y_{2}))}{a} \in
  \left\{ \frac{Q(x,y)}{a} : x, y \in \Q \right\}. \qedhere
\end{equation*}
\end{proof}

Next we require an analogue of Lemma \ref{quotset} that describes the $p$-adic closure $R(Q)^-$ of $R(Q)$.

\begin{lemma}\label{Lemma:CloseQp}
If $a$ is a nonzero integer represented by $Q$, then
\begin{equation*}
R(Q)^-  = \left\{ \frac{Q(x,y)}{a} : x, y \in \Q_{p} \right\}.
\end{equation*}
\end{lemma}

\begin{proof}
Suppose that $b = Q(x,y)/a$ with $x, y \in \Q_{p}$. Write
$a = Q(z,w)$ with $z, w \in \Q_{p}$ and choose sequences $x_n,y_n$ of rational numbers such that
\begin{equation*}
\lim_{n \to \infty} x_n = x \qquad \text{and} \qquad \lim_{n \to \infty} y_n = y
\end{equation*}
in $\Q_p$.  The continuity of $Q$ ensures that
\begin{equation*}
\frac{Q(x,y)}{a} = \frac{\lim_{n \to \infty} Q(x_{n},y_{n})}{a} = \lim_{n \to \infty} \frac{Q(x_{n},y_{n})}{a},
\end{equation*}
so $Q(x,y)/a$ is a limit point of $R(Q)$ by Lemma~\ref{quotset}.
Thus, $Q(x,y)/a \in R(Q)^-$.

Now suppose that $b \in R(Q)^- $. If $b = 0$, then $b = Q(0,0) / a$ and we are done.
If $b \neq 0$, Lemma \ref{quotset} provides $x,y \in \Q$ such that
\begin{equation*}
\left\| b - \frac{Q(x,y)}{a} \right\| \leq \frac{\|b\|}{p^{3}},
\qquad\text{and hence}\qquad
  \left\| 1 - \frac{Q(x,y)}{ab} \right\| \leq \frac{1}{p^{3}},
\end{equation*}
which implies that
\begin{equation*}
1- \frac{Q(x,y)}{ab} \in \Z_{p}
\qquad \text{and} \qquad 
\frac{Q(x,y)}{ab} \equiv
1 \pmod{p^{3}}.
\end{equation*}
Since every element of $\Z_{p}$ that is congruent to
$1$ modulo $p^{3}$ is a square (by \cite[Thms. II3 \& II.4]{MR0344216} mentioned above), there is a $w \in \Z_{p}$ such
that $Q(x,y)/(ab) = w^{2}$. Then
\begin{equation*}
b = \frac{Q(x,y)}{aw^{2}} = \frac{Q(x/w,y/w)}{a} \in \left\{ \frac{Q(x,y)}{a} : x, y \in \Q_{p} \right\}. \qedhere
\end{equation*}
\end{proof}

We can now reprove Theorems~\ref{Theorem:Nonsingular},
\ref{Theorem:Singular}, and \ref{Theorem:Two}. Lemma
\ref{Lemma:CloseQp} implies that the $p$-adic closure of $R(Q)$
depends only on the $\Q_{p}$-equivalence class of $Q$.  A quadratic
form over a field can be diagonalized, and so up to scaling, any
binary quadratic form is equivalent to $Q(x,y) = x^{2} - dy^{2}$,
where $d$ is a representative of the $\Q_{p}$-square class of the
discriminant of $Q$. As mentioned earlier, the corollary on page 37 of
\cite{MR0344216} shows that $Q$ represents every element of $\Q_{p}$
if and only if $d$ is a square in $\Q_{p}$. For this reason,
$R(Q)$ is dense in $\Q_{p}$ if and only if the discriminant of $Q$ is a square
in $\Q_{p}$. In particular, if $p > 2$ and $b^{2} - 4ac = p^{k} \ell$,
then $R(Q)$ is dense in $\Q_{p}$ if and only if $k$ is even and $(\ell/p) = 1$.
If $p = 2$, and $b^{2} - 4ac = 2^{k} \ell$, then $R(Q)$ is dense in $\Q_{2}$
if and only if $k$ is even and $\ell \equiv 1 \pmod{8}$.

Now, we turn to the situation of quadratic forms in $r \geq 3$ variables.
Suppose that $Q(\vec{x}) = x^{T} A \vec{x}$ is an integral quadratic form
in $r \geq 3$ variables and $\det(A) \ne 0$.  The special case $A = I$ and $r=3$ was settled
by Miska, Murru, and Sanna \cite[Thm.~1.8c]{MiskaMurruSanna}.

\begin{theorem}
If $r \geq 3$, then $R(Q)$ is dense in $\Q_{p}$ for all primes $p$.
\end{theorem}
\begin{proof}
  Fix an $n \in \Q_{p}$. If $n = 0$, then it is clear that $n$ is in the $p$-adic closure of $R(Q)$, since we can take a vector $\vec{y} \in \Z^{r}$
  so that $Q(\vec{y}) \ne 0$, and note that $0 = \frac{Q(\vec{x})}{Q(\vec{y})} \in R(Q)$.
  
  Assume therefore that $n \ne 0$. By the same corollary from page 37
  of \cite{MR0344216} quoted above, the forms $Q$ and $nQ$ each
  represent either everything in $\Q_{p}$ or all but one square class
  in $\Q_{p}$. Since $\Q_{p}$ has four square classes if $p > 2$ (and
  eight if $p = 2$), there must be some nonzero element $d \in \Q_{p}$
  represented by both $Q$ and $nQ$. By scaling these representations
  by a power of $p$, we can assume that there are vectors $\vec{x} \in
  \Z_{p}^{r}$ and $\vec{y} \in \Z_{p}^{r}$ so that $Q(\vec{x}) =
  nQ(\vec{y}) = k$ with $k \in \Z_{p}$ and $k \ne 0$.

  Fix $\epsilon > 0$. Since $\Z$ is dense in $\Z_{p}$, there are
  vectors $\vec{z} \in \Z^{r}$ and $\vec{w} \in \Z^{r}$ (with
  components $z_{1}$, $\ldots$, $z_{r}$ and $w_{1}$, $\ldots$,
  $w_{r}$) so that $$\|z_{i} - w_{i}\| < \delta := \min\{\epsilon
  \|k/n\|, \|k/n\|, \epsilon \|k/n^{2}\|\}$$ for all $i$ (and similarly
  $\|y_{i} - w_{i}\| < \delta$ for all $i$). Since $Q$ is a polynomial
  with integer coefficients, $Q$ is $p$-adically continuous. In fact,
  the ultrametric inequality implies that if $a_{1}, \ldots, a_{r}$
  and $b_{1},\ldots,b_{r}$ are elements of $\Q_{p}$ with
  $\|a_{i}-b_{i}\| < \epsilon$ for all $i$, then
\[
  \|Q(a_{1},a_{2},\ldots,a_{r}) - Q(b_{1},b_{2},\ldots,b_{r})\| < \epsilon.
\]
Using this, we have that
\begin{align*}
\|Q(\vec{z}) - nQ(\vec{w})\| &=
\|Q(\vec{z}) - Q(\vec{x}) + Q(\vec{x}) - nQ(\vec{y}) + nQ(\vec{y}) - nQ(\vec{w})\|\\
&\leq \max(\|Q(\vec{z}) - Q(\vec{x})\|,\|Q(\vec{x}) - nQ(\vec{y})\|
,\|nQ(\vec{y}) - nQ(\vec{w})\|)\\
&< \max(\epsilon \|k/n\|, 0, \|n\| \epsilon \|k/n^{2}\|) \leq \epsilon \|k/n\|.
\end{align*}
Since $\|Q(\vec{w}) - Q(\vec{y})\| < \|k/n\|$ and $Q(\vec{y}) =
k/n$, it follows that $\|Q(\vec{w})\| = \|Q(\vec{y})\| =
\|k/n\|$.  Thus,
\begin{align*}
\left\| \frac{Q(\vec{z})}{Q(\vec{w})} - n\right\|
&= \frac{1}{\|Q(\vec{w})\|} \cdot \| Q(\vec{z}) - nQ(\vec{w}) \|
= \frac{1}{\|Q(\vec{y})\|} \cdot \| Q(\vec{z}) - nQ(\vec{w}) \|\\
&< \frac{1}{\|k/n\|} (\epsilon \|k/n\|) < \epsilon.
\end{align*}
This proves that $n$ is in the $p$-adic closure of $R(Q)$, as desired.
\end{proof}

\bibliographystyle{amsplain}
\bibliography{pQS4}

\end{document}